\documentclass[oneside,english]{amsart}
\usepackage{cmap}
\usepackage[latin1]{inputenc} 
\usepackage{setspace} 

\usepackage{amssymb}
\usepackage[]{epsf,epsfig,amsmath,amssymb,amsfonts,latexsym}
\usepackage{comment}
\usepackage{enumerate}

\usepackage{tikz}
\usetikzlibrary{arrows,decorations.markings}
\usetikzlibrary{calc}

\usepackage{fullpage}
\usepackage{color}
\usepackage{cancel}
\usepackage{amsmath}
\usepackage[normalem]{ulem}
\usepackage{amsfonts}
\usepackage{bbm}
\usepackage{amsthm}
\usepackage[mathscr]{eucal}
\usepackage[active]{srcltx}
\usepackage{fancyhdr}
\usepackage{verbatim}

\usepackage{hyperref}
\usepackage{mathtools}
\hypersetup{
colorlinks = true, 
urlcolor = blue, 
linkcolor = red, 
citecolor = blue 
}

\def\N{\mathbb N}

\newcommand{\Z}{\mathbb{Z}}

\def \E{\mathbb E}
\def \R{\mathbb R}

\def \E1{\mathcal E}

\def\zd{{\mathbb Z^d}}

\def \H{\mathcal H}

\newcommand{\supp}{\text{Supp}}
\newcommand{\Pred}{\text{\tt{Predict}}}
\newcommand{\SIPp}{\text{SIP}^+}

\newenvironment{dedication}{\begin{quotation}\vspace{2 pt}\begin{center}\begin{em}}{\par\end{em}\end{center}\vspace{2 pt}\end{quotation}}
\newtheorem{proposition}{Proposition}
\newtheorem{theorem}{Theorem}
\newtheorem{lemma}{Lemma}
\newtheorem{corollary}{Corollary}
\newtheorem*{definition}{Definition}

\newtheorem*{example}{Example}
\newtheorem{question}{Question}

\def\N{\mathbb N}

\def \E{\mathbb E}
\def \R{\mathbb R}

\def \E1{\mathcal E}

\def \SIP{\textit{SIP}_+}

\def \H{\mathcal H}

\def\SIP{\text{SIP}^\star}
\def\0{\mathbf 0}
\theoremstyle{remark}
\newtheorem{remark}{Remark}
\title{Predictive Sets}

\author{
	Nishant Chandgotia
	\and
	Benjamin Weiss
}

\address{Einstein Institute of Mathematics, Hebrew University of Jerusalem, Israel}
\email{nishant.chandgotia@gmail.com}
\address{Einstein Institute of Mathematics, Hebrew University of Jerusalem, Israel}
\email{weiss@math.huji.ac.il}
\subjclass[2010]{Primary 60G25; Secondary 37A05}
\keywords{zero entropy processes, predictive sets, SIPs, return-times sets, Gaussian processes, Riesz sets}

\begin{document}
	\large
		\maketitle
	\begin{dedication}
		Dedicated to our dear friend and colleague Manfred Denker on his diamond jubilee.
	\end{dedication}
		\begin{abstract}
		A set $P\subset \N$ is called predictive if for any zero entropy finite-valued stationary process $(X_i)_{i\in \Z}$, $X_0$ is measurable with respect to $(X_{-i})_{i\in P}$. We know that $\N$ is a predictive set. In this paper we give sufficient conditions and necessary ones for a set to be predictive. We also discuss linear predictivity, predictivity among Gaussian processes and relate these to Riesz sets which arise in harmonic analysis.
	\end{abstract}

\section{Introduction}
A deterministic stationary stochastic process $\{X_n\}_{n\in \Z}$ is a process for which $X_0$ is a measurable function of the past $\{X_{-n}\}_{n\in \N}$. If the process is finite-valued as well then this is equivalent to it having zero entropy. If a process has zero entropy so does the process formed by sampling along an arithmetic progression and so $H(X_0~|~X_{-nk}; n\in \N)=0$ for all positive integers $k$. In other words, $X_0$ is a measurable function of $\{X_{-kn}\}_{n\in \N}$ for all $k\in \N$. This motivates the following definition.
\begin{definition}
	A subset $P\subset \N$ is predictive if for all finite-valued zero entropy stationary processes $\{X_n\}_{n\in \Z}$ we have that
	$$H(X_0~|~X_{-p}; p\in P)=0.$$
\end{definition}
The aim of this note is to explore $\Pred$, the set of all predictive sets. While we have not found a characterisation of predictive sets, we have found some necessary conditions and some sufficient ones. There are many open questions which we have been unable to answer and are interspersed throughout our discussion.

Here are some of the main results that we found.

\smallskip
\noindent\textbf{Sufficient conditions: }If $(X, \mathcal B, \mu, T)$ is a probability preserving transformation and $U\subset X$ is a set of positive measure then $\{n\in \N~:~\mu(U\cap T^{-n}(U))>0\}$ is called the return-time set for $U$. Return-time sets are predictive.

\smallskip
\noindent\textbf{Necessary conditions: }A predictive set has bounded gaps. This follows from a stronger statement that it must intersect non-trivially all SIPs (symmetric infinite parallelepipeds). These are sets of the form $$\{\sum_i \epsilon_ip_i~:~\epsilon_i\in \{-1,0,1\}\}$$ where $\{p_i\}_{i\in \N}$ is a sequence of natural numbers.

We will also discuss linear prediction, which means that $X_0$ is in the closed linear space spanned by $\{X_{-n}~:~n\in \N\}$ where the closure is taken in the $L^2$-norm. This of course is tied to the theory of Gaussian processes which will also play a key role in some of the proofs. 

Almost all of the results that we have found can be readily extended to stationary random fields indexed by $\zd$ with an appropriate notion of the past. 

Note that the fact that $P$ is a predictive set doesn't mean that given $X_{-p}; p \in P$ we know the entire future $X_i; i\geq 0$ as is the case for $P=\N$. For this we define a stronger property: Totally predictive sets.

\begin{definition}
A set $P\subset \N$ is \emph{totally predictive} if for all zero entropy stationary processes $\{X_n\}_{n\in \Z}$ and $q\in \Z$ we have that 
$$H(X_q~|~X_{-p};p \in P)=0.$$
\end{definition}
We prove that if $P$ is predictive then for all singular probability measures $\mu$ on the circle we must have that $\hat\mu(p)\neq 0$ for some $p\in P$. We also show that if $P$ is totally predictive then for all singular complex-valued measures $\mu$ on the circle we must have that 
$\hat\mu(p)\neq 0$ for some $p\in P$. Such sets are very closely related to Riesz sets which we define next.
\begin{definition}
A set $Q\subset \Z$ is a \emph{Riesz set} if all measures $\mu$ on the circle satisfying
$\hat\mu (q)=0$ for $q\in \Z\setminus Q$ are absolutely continuous.
\end{definition}
These sets were defined by Yves Meyer in \cite{MR227697} and extensively studied over the years. While we could not prove an implication either way, we find that the techniques developed to prove that a set $Q\subset \Z$ is Riesz help us prove that the set $\N\setminus Q$ is totally predictive for many examples. As a partial result in this direction we show that if $\N\setminus Q$ is totally predictive and open in the Bohr topology restricted to $\N$ (defined later in Section \ref{section:Totally predictive}) then $Q\cup (\Z\setminus \N)$ is a Riesz set.

The paper is structured in the following way. The notation used in the paper is given in Section \ref{Section: Notation}. In Section \ref{section:return-time sets} we prove that return-time sets are predictive and describe many of its consequences. In Section \ref{Section:SIP} we give necessary conditions for a set to be predictive, namely that it should be an $\SIP$. In Section \ref{Section:Comments Gaussian} we talk about linear prediction and some consequences for Gaussian processes. In Section \ref{section:Totally predictive} we study totally predictive sets and relate them to Riesz sets. In Section \ref{section: some examples} we give several examples of totally predictive sets. In Section \ref{section:squares are sparse} we show that there are totally predictive sets which are not return-time sets and remark about the total predictivity of the complements of certain very sparse sets. Finally in Section \ref{section:zd} we discuss predictivity for $\zd$ actions. 

\section*{Acknowledgements}
The first author has been funded by ISF grants 1702/17, 1570/17 and ISF-Moked grants 2095/15, 2919/19. We would like to thank Hillel Furstenberg, Eli Glasner, Jon Aaronson, Vitaly Bergelson and Mikhail Sodin for several discussions. We thank KV Shuddhodhan for explaining some number theoretic facts to us and Elon Lindenstrauss for his suggestion which lead to the proof of Proposition \ref{prop:elon} and its corollary. We are very grateful to Herve Queffelec for introducing us to Riesz sets. Finally we would like to thank John Griesmer and the anonymous referee for a careful reading and many helpful comments.

\section{Notation}\label{Section: Notation}
In this note a \emph{process} will always mean a finite-valued (unless otherwise mentioned) stationary process. $\N$ will denote the natural numbers $\{1,2, \ldots\}$ and $\Z$ will denote the integers.

 For a process $(X_i)_{i\in \Z}$ and $P\subset \Z$, let $X_P$ denote the complete sigma-algebra generated by the collection of random variables $(X_i)_{i\in P}$. Abusing notation we will also use $X_\Z$ to denote the process itself. The entropy of a process can be computed by the formula
$$h(X_\Z):=H(X_0~|X_{\N}).$$
This motivates the following definition.
\begin{definition}
 A set $P\subset \N$ is \emph{predictive} if for all zero entropy processes, $X_\Z$, $X_0$ is measurable with respect to $X_{P}$.
\end{definition}

Though one ought to talk about prediction using the past, for notational convenience we will ``predict using the future'' keeping in mind that $(X_i)_{i\in \Z}$ is a zero entropy process if and only if $(X_{-i})_{i\in \Z}$ is as well.

The formula for the entropy of a process implies that $\N$ is a predictive set. Since restricting a zero entropy process to multiples of a natural number $k$ yields a zero entropy process, $k\N$ is also a predictive set. However it is not difficult to see that for natural numbers $r$ which are not a multiple of $k$, $k\N +r$ is not predictive. Let $\zeta_0, \zeta_1, \ldots, \zeta_{k-1}$ be independent random variables with $\zeta_i$ taking the value $-1$ or $1$ with probability $1/2$ each. Let $Z$ be the uniform random variable on $\{1, 2, \ldots, k\}$ independent of $\zeta_i$ for $0\leq i \leq k-1$. Now consider the process $X_\Z$ given by 
$$X_{kn+r}:= (Z\oplus r)\zeta_r\text{ for all }n\in \Z\text{ and }r\in \{0,1,\ldots, k-1\}$$
where $Z\oplus r:= s$ for $s\in \{1, 2, \ldots, k\}$ for which 
$Z+r\equiv s\mod k$. We have that for $r$ which is not a multiple of $k$, $H(X_0~|~X_{k\N+r})=\log (2)$ even though $X_\Z$ is a zero entropy process.

\begin{remark}
	In a similar fashion, one can prove that for all $k\in \N$, $P$ is predictive if and only $k P$ is also predictive. Suppose $P$ is predictive. Given a zero entropy process $X_\Z$ we know that $X_{k\Z}$ also has zero entropy. Applying the predictivity of $P$ to this new process we have that $H(X_0~|~X_{kP})=0$. 
	
	Now suppose that $kP$ is predictive and $X_\Z$ is a given zero entropy process whose state space does not contain $0$. Let $W$ be an independent uniform random variable on $\{0,1,2,\ldots, k-1\}$ and let $Y_\Z$ be the zero entropy process given by $Y_{i+kn}:=X_n1_{W}(i)$ for all $n\in \Z$ and $0\leq i \leq k-1$.
Since $kP$ is predictive, $H(Y_0~|~Y_{kP})=0$. Thus $H(Y_0~|~Y_{kP}, W=0)=0$. But $H(Y_0~|~Y_{kP}, W=0)=H(X_0~|~X_P)$.
Hence $P$ is predictive.
\end{remark}

\begin{remark}
	Predictive sets form a \emph{family}, meaning that if $P$ is predictive and $P\subset Q\subset \N$ then $Q$ is also predictive: Clearly $H(X_0~|~X_{Q})\leq H(X_0~|~X_{P})$ and hence if the latter is zero then the former is zero as well.
\end{remark}

\begin{remark}
	For most of the paper we will focus on finite-valued processes. One of the reasons is that for infinite-valued processes the correspondence between zero entropy and predictivity of $\N$ does not exist.
	
	Indeed take any process $X_\Z$ with state space $\{0,1\}$ and let $\{j_1, j_2, \ldots\}$ be an enumeration of $\Z$.  The process $Y_\Z$ with state space $[0,1]$ given by 
	$Y_i:= \sum_k \frac{1}{3^k}X_{i+j_k}$
	is isomorphic to $X_\Z$ and $X_\Z$ is measurable with respect to $Y_{1}$ and hence $Y_0$ is measurable with respect to $Y_\N$.
\end{remark}

\section{Predictivity of return-time sets}\label{section:return-time sets}

	We will assume through out that $X$ is a Polish space. A set $A\subset \N$ is \emph{a return-time set}  if there exists a probability preserving transformation (\emph{ppt}) $(X, \mu, T)$ and a set $U\subset X; \mu(U)>0$ such that 
	$$A=N(U,U):=\{n\in \N~:~ \mu(U\cap T^{-n}(U))>0\}.$$

One of the main results of this section is that return-time sets are predictive. We will need the following notation: 	Given a set $Q\subset \N$ let 
$$Q_n:=Q\cap\{1,2,\ldots, n\}.$$ The \emph{upper density} of $Q$ is defined as
$$\limsup_{n\to \infty}\frac{|Q_n|}{n}.$$

\begin{theorem}\label{thm:difference_set_predictive}
	Let $Q$ be a set with positive upper density. Then $(Q-Q)\cap \N$ is predictive.
\end{theorem}

\begin{proof}
	Let $Q\subset \N$ be a set of positive upper density. Then we have that 
	$$\frac{1}{n}H(X_{Q_n})\leq \frac{1}{n}H(X_{\N_n})$$
	proving that $\frac{1}{n}H(X_{Q_n})$ converges to zero as $n$ tends to infinity. Since the set $Q$ has positive upper density there exists an increasing sequence of natural numbers $n_k$ such that 
	$$\lim_{k \to \infty}\frac{|Q_{n_k}|}{{n_k}}>0\text{ and hence }\frac{1}{|Q_{n_k}|}H(X_{Q_{n_k}})= \frac{n_k}{|Q_{n_k}|}\frac{1}{{n_k}}H(X_{Q_{n_k}})$$
	converges to zero as $k\to \infty$. 
	Write $Q_{n_k}:=\{q_1, q_2, \ldots, q_{m_k}\}$ where (as follows) $m_k=|Q_{n_k}|$ and $q_1<q_2<\ldots<q_{m_k}$.
	By the chain-rule and the stationarity of the process we get that
	\begin{eqnarray*}
		H(X_{Q_{n_k}})&=& \sum_{i=1}^{m_k}H(X_{q_{i}}~|~X_{q_{i+1}},X_{q_{i+2}}, \ldots, X_{q_{m_k}})\\
		&=&\sum_{i=1}^{m_k}H(X_0~|~X_{q_{i+1}-q_{i}},X_{q_{i+2}-q_i}, \ldots, X_{q_{m_k}-q_i})\\
		&\geq&m_k H(X_0~|~X_{(Q-Q)\cap {\N}})	\end{eqnarray*}
	Thus $$\frac{1}{|Q_{n_k}|}H(X_{Q_{n_k}})\geq H(X_0~|~X_{(Q-Q)\cap {\N}}).$$
Since the left hand side converges to zero we get that $(Q-Q)\cap {\N}$ is a predictive set.
\end{proof}

Now we are ready to prove that return-time sets are predictive. 
\begin{corollary}\label{cor:return time predictive}
	Return-time sets are predictive.
\end{corollary}
In the following proof, we will need the \emph{shift map} $\sigma: \{0,1\}^\Z\to \{0,1\}^\Z$ given by 
$(\sigma(x))_i:=x_{i+1}$. In addition we will denote by 
$$[1]:=\{x\in \{0,1\}^\Z~:~x_0=1\},$$
the cylinder set of elements of the shift space with $1$ at the origin.
\begin{proof}
We know from Theorem \ref{thm:difference_set_predictive} that the difference set of a set of positive upper density is predictive. So it is sufficient to prove that every return-time set $P$ contains  $(Q-Q)\cap \N$ for a set $Q$ of positive upper density.

Let $P$ be a return-time set. Then there exists a ppt $(X, \mu, T)$ and $U\subset X; \mu(U)>0$ such that $N(U, U)=P$. Consider the map $\phi:(X,T)\to (\{0,1\}^\Z,\sigma)$ given by
$$\phi(x)_i=1_U(T^i(x)).$$
Let $\nu$ be the push-forward of the measure $\mu$ and let $Y\subset \{0,1\}^\Z$ be the topological support of $\nu$. It follows that $N(U,U)=N([1],[1])$. By the Birkhoff ergodic theorem there exists $y\in Y$ such that
$$\lim_{n\to \infty}\frac{1}{n}\sum_{i=0}^{n-1}1_{[1]}(T^i(y))>0.$$
Thus the set $Q=\{i\in \N~:~y_i=1\}$ has positive density. On the other hand since $y$ is in the support of $\mu$ we have that if 
$y_i=y_j=1$ for some $i,j\in \N$ then $\nu(T^{-|i-j|}([1])\cap [1])>0$. Hence we get that $$(Q-Q)\cap \N\subset N([1], [1])= N(U,U).$$
This completes the proof.
 \end{proof}

\begin{remark}
In Proposition \ref{prop:return not predictive} we construct a (totally) predictive set which does not contain any return-time set. 
\end{remark}

\begin{remark}\label{remark: entropy_bound}
	The proof of Theorem \ref{thm:difference_set_predictive} can be broken up into two parts. In the first part we showed that if a set $Q$ has positive upper density $c>0$ then
	$$\liminf \frac{1}{|Q_n|}H(X_{Q_n})\leq \frac{1}{c} h(X_\Z)$$
	and in the second part we used the chain-rule for entropy to show that
	$$h(X_0~|~X_{(Q-Q)\cap {\N}})\leq \liminf \frac{1}{|Q_n|}H(X_{Q_n}).$$
	These inequalities hold for all processes and not just zero entropy ones. Thus we have in fact that if $P$ is a return-time set then there exists a constant $c>0$ such that
	$$H(X_0~|~X_{P})\leq \frac{1}{c}\  h(X_\Z)$$
	for all processes $X_\Z$.

\end{remark}
	\begin{question}
	Suppose $P$ is a predictive set. Does there exists $c>0$ such that for all processes $X_\Z$,
	$$H(X_0~|~X_{P})\leq c\  h(X_\Z)?$$
\end{question}

\begin{remark}
Given a set $Q\subset \N$ the \emph{sequence entropy }along $Q$ is calculated by 
$$h(X_Q):=\limsup_{n\to \infty} \frac{1}{|Q_n|}H(X_{Q_n}).$$
In \cite{MR0322136} Krug and Newton showed that there is a constant $K(Q)\in [0,\infty]$ such that 
$h(X_Q)= K(Q)h(X_\Z)$ where 
$0. \infty =0$ and $\infty .0$ is undefined. In view of Remark \ref{remark: entropy_bound} if $K(Q)<\infty$ then $(Q-Q)\cap \N$ is predictive. However it is not difficult to see that any such set must contain a return-time set and thus this does not give us any additional information. On the other hand if a process has discrete spectrum then $h(X_Q)=0$ for all infinite sets $Q\subset \N$ \cite{Kushnirenko_1967}. This implies that for any infinite set $Q$, $(Q-Q)\cap \N$ is predictive for processes with discrete spectrum.

\end{remark}

\begin{question}
	The ergodic decomposition shows that the family generated by return-time sets is the same as that generated by return-time sets for ergodic ppts. Is it the same as the return-time sets for zero entropy ppts?
\end{question}

By a construction of K\v{r}\'{\i}\v{z} \cite{MR932131} we know that there are return-time sets which do not contain the return-time sets for ppts with discrete spectrum (see also \cite{mccutcheon_three_1995}). The motivation for this question comes from the following proposition.

\begin{proposition}\label{prop:zero_entropy_intersect}
	If $P$ is predictive and $Q$ is a return-time set of a zero entropy process then $P\cap Q$ is predictive.
\end{proposition}

\begin{proof}
Let $P$ be a predictive set, $(Y,\mathcal C, \mu, T)$ be a zero entropy ppt and $U \subset Y$ have positive measure. We will prove that $P\cap N(U, U)$ is still predictive. Let  $X_\Z(\omega)$ be a zero entropy process on the probability space $(\Omega,\mathcal B, \nu)$ and assume that $0$ is not in its state space. We now consider the zero entropy process $W_\Z$ defined on $(Y\times \Omega, \mathcal{C}\times \mathcal{B}, \mu\times \nu)$ and given by 
$$W_i(y,\omega):=X_i(\omega) 1_{U}(T^i(y)).$$
Since $P$ is predictive,
$$H(X_0 1_{U}(y)~|~ X_{i}1_{U}(T^{i}(y)); i \in P)=0$$
and thus
$$H(X_0 ~|~ X_{i}1_{U}(T^{i}(y)); i \in P,  y \in U)=0.$$
From this it follows that 
$$H(X_0 ~|~ X_{i}; i \in P\cap N(U,U),(1_{ U}(T^{i}(y)))_{i \in P}, y\in U)=0.$$
The distribution of $y$ and hence the distribution of $T^i(y)$ is independent of $X_\Z$. Thus
$$H(X_0~|~X_{P\cap N(U,U)})=0$$ proving that
$P\cap N(U,U)$ is predictive. 
\end{proof}
This motivates a further question  (see also Question \ref{question:filter or sip}).
\begin{question}\label{Question:Filter}
Do predictive sets form a \emph{filter}, that is, if $P$ and $Q$ are predictive then is $P\cap Q$ predictive as well? 
\end{question}
Notice that return-time sets do form a filter. Given a dynamical system $(X,T)$, an open set $U$ and a point $x\in X$ the \emph{visit-time sets} $N(x, U)$ are given by
$$N(x, U):=\{n\in \N~:~ T^n(x)\in U\}.$$ 
\begin{corollary} \label{corollary:predictive_rotation}
Let $P$ be a predictive set and $(G,R_\alpha)$ be a compact group rotation by $\alpha\in G$. For open sets $U\subset G$ and $x\in U$, $P \cap N(x,U)$ is predictive. 
\end{corollary}

\begin{proof} 
	In view of Proposition \ref{prop:zero_entropy_intersect} it is enough to prove that $N(x, U)$ contains a return-time set of a zero entropy ppt. Consider rotation by $\alpha$ on a group $G$, $U$ an open set containing the identity $e$. Choose a symmetric open neighbourhood $V$ of $e$ (meaning $V=V^{-1}$) such that $V.V\subset U$. $V$ is a positive measure set for the Haar measure $\mu$ on $G$ and 
	$$ N(V,V)=\{n\in \N~:~\mu(\alpha^nV\cap V)>0\}=\{n \in \N~:~\alpha^n V\cap V\neq \emptyset\}.$$
A simple algebraic manipulation using the fact that $V$ is symmetric gives us that
$$N(V,V)\subset N(e, V.V)\subset N(e, U).$$ Since compact group rotations have zero entropy this proves that $P\cap N(e, U)$ is predictive. The corollary now follows from the additional observation that for any $x\in G$ and open sets $W\subset G; x\in W$, we have that $N(x, W)=N(e, x^{-1}(W))$ and $x^{-1}(W)$ is an open set containing the identity.
\end{proof}

We will write $\mathbb T:=\R/\Z$ to denote the circle. In view of this corollary it is natural to ask about visit-times for distal systems. Consider for instance the distal system: $T:\mathbb T^2\to \mathbb T^2$ given by $T(x, y):= (x+\alpha,y+2x+\alpha)$. The forward orbit of the origin $(0,0)$ is $(n\alpha,n^2 \alpha)_{n \in \N}$. 

\begin{question}
	Let $\alpha\in \mathbb T$ be irrational and $\epsilon>0$. Is the set $\{n\in \N~:~|n^2\alpha|<\epsilon\}$ predictive?
\end{question}

Now moving in a different direction let us see why return-time sets don't just predict $X_0$ but also $X_n$ for a set of times $n \in \Z$ which have bounded gaps. 
\begin{proposition}\label{proposition:bounded gaps}
Let $P$ be a return-time set. Then the set 
$$\{n\in \Z~:~ P+n\text{ contains a return-time set}\}$$
has bounded gaps.
\end{proposition}
\begin{proof}
Let $(X, \mu, T)$ be a ppt and $U\subset X$ be a set with positive measure. Let us first see that if $n\in N(U, U)$ then both $n+N(U,U)$ and $-n+N(U,U)$ contain a return-time set: Consider the transformation 
$$(X\times \Z/2n\Z, \mu\times \mu_H, T')$$
where $\mu_H$ is the uniform measure on $\Z/2n\Z$ and 
$$T'(x,y):=(T(x),y+1\mod 2n).$$
Let $V:=(U\cap T^{-n}(U))\times \{0\}$.
If $r\in N(V, V)$ then $r\geq 2n$ and 
$$\mu (U \cap T^{-n}(U)\cap T^{-r}(U)\cap T^{-n-r}(U))>0$$
which implies that both $\mu(U \cap T^{-n-r}(U))$ and $\mu(U \cap T^{-r+n})$ are positive.  Thus $N(V,V)$ is contained in both $-n+ N(U,U)$ and $n+N(U, U)$.

We are left to prove that the set $N(U, U)$ has bounded gaps. This is easy to see from the Poincar\'e recurrence theorem: Let $W:=\cup_{n\geq  0}T^{-n}(U)$ be the set swept by $U$. Choose $r\in\N$ large enough such that $\mu(W\setminus(\cup_{n=0}^rT^{-n}(U)))<\mu(U)/2$. Since $T$ is a ppt we have that
$$\mu(W\setminus(\cup_{n=k}^{r+k}T^{-n}(U)))<\mu(U)/2$$ for all $k\in \N$.

But if $\{n\in \N~:~\mu(U\cap T^{-n}(U))>0\}$ has a gap of length $t$ then $\mu(U)<\mu(W)-\mu(\cup_{n=0}^{t-1} T^{-n}(U))$ proving that $t\leq r$. Thus the set $N(U, U)$ has bounded gaps.
\end{proof}

\begin{corollary}
	Let $P$ be a return-time set. Then the set 
$$\{n\in \Z~:~(P+n)\cap \N\text{ is predictive}\}$$
has bounded gaps.
\end{corollary}

\begin{proof}
By Proposition \ref{proposition:bounded gaps} the set 
$$\{n\in \Z~:~ P+n\text{ contains a return-time set}\}$$
has bounded gaps. For such $n\in \Z$, since there are return-time sets $Q\subset P+n$ it follows from Proposition \ref{prop:zero_entropy_intersect} that 
$(P+n)\cap \N$ is predictive. 
	\end{proof}
The following question arises naturally.
\begin{question}
Let $P$ be a predictive set. Does the set 
$$\{n\in \Z~:~(P+n)\cap \N\text{ is predictive}\}$$
have bounded gaps?
\end{question}

As a final suggestion which one may pursue,  we consider positive definite sequences: A \emph{positive definite sequence}  is a sequence $a_{n}; {n\in \Z}$ of complex numbers such that for all sequences $b_n$ and natural numbers $N$, 
$$\sum_{n,m=1}^N a_{n-m}b_n\overline{b_m}\geq 0.$$

It is not difficult to see that for a ppt $(X, \mu, T)$ and a set $A\subset X$ of positive measure, the sequence $\mu(A\cap T^{-n}(A)); n \in \N$ is a positive definite sequence and the return-time set is just the support of the sequence in the natural numbers. Herglotz \cite[Page 90]{MR614142} proved that $a_n; n \in \Z$ is positive definite if and only if there is a finite non-negative measure $\mu$ on the circle $\mathbb T$ such that its Fourier coefficients
$$a_n=\hat{\mu}(n)=\int_{\mathbb T} e^{-2\pi i nt}d\mu(t).$$
The following question arises naturally.
\begin{question}\label{Question:positive definite}
Fix $\epsilon>0$ and a real valued positive definite sequence $a_n; n\in \Z$. Is the set 
$$\{n\in \N~:~a_n >-\epsilon\}$$
predictive?
\end{question}

\begin{remark}
Given a ppt $(X, \mu, T)$ and $A\subset X$ of positive measure and the positive-definite sequence $a_n:=\mu(A\cap T^{-n}(A))-\mu(A)^2$ we are just asking whether the set
$$\{n\in \N~:~\mu(A\cap T^{-n}(A))>\mu(A)^2-\epsilon\}$$
is predictive. Khintchine's theorem \cite[Page 22]{MR614142} shows that the set has bounded gaps.
\end{remark}

\begin{remark}
	We cannot take $\epsilon=0$ in this question: Let $\mu$ be the Lebesgue measure supported on $[-1/4, 1/4]\subset \mathbb T$. The sequence 
$$a_n:=\int_{\mathbb T} e^{2\pi i n x} d\mu(x)= \frac{1}{\pi n} sin(\pi n/2) .$$
is positive definite but its support is the set of odd numbers which is not predictive.
\end{remark}
While we do not know the answer to Question \ref{Question:positive definite} we will indicate a result in favour of these sets being predictive in Proposition \ref{prop:bounded gaps of positive definite sequecnes}. However we cannot hope to use Corollary \ref{cor:return time predictive} to answer Question \ref{Question:positive definite} since a construction by Forrest shows that these sets need not contain return-time sets \cite{MR1177341} (see also \cite{mccutcheon_three_1995}).

\section{Predictive sets are $\SIP$}\label{Section:SIP}

 In the last section we gave sufficient condition for a set to be predictive. In this section, we will investigate necessary conditions for predictivity. We will show that predictive sets have bounded gaps.

To prove this we will show something stronger which we now describe. Given a sequence of natural numbers $s_i;i \in \N$ we denote by 
$$\SIPp(s_1, s_2, \ldots):=\{\sum_{i\in \N}\epsilon_i s_i~:~ \epsilon_i=0,1,-1\}\cap \N,$$ 
the \emph{symmetric infinite dimensional parallelepiped set (SIP)} generated by it. 

A \emph{$\SIP$} set is a subset of the natural numbers which intersects every SIP set non-trivially. 

\begin{theorem}\label{thm:SIpstar}
	Predictive sets are $\SIP$. 
\end{theorem}
\begin{remark}
Infinite dimensional parallelepiped sets or IP sets are like SIP sets except that $\epsilon_i$ are either $0$ or $1$. We will show in Theorem \ref{thm:some totally predictive sets} that predictive sets need not be $\text{IP}^\star$, that is, the complement of IP sets can be predictive.
\end{remark}

It is well-known that $\SIP$ sets have bounded gaps. Thus if we prove this theorem we will have proved that predictive sets have bounded gaps. We give a proof of this corollary for completeness.

\begin{corollary} Predictive sets have bounded gaps.\label{cor: bounded gaps predictive}
\end{corollary}
Given a finite sequence of natural numbers $s_1, s_2, \ldots, s_n$ we write 
$$\SIPp(s_1, s_2, \ldots, s_n):=\{\sum_{i=1}^n\epsilon_i s_i~:~ \epsilon_i=0,1,-1\}\cap \N.$$ 
\begin{proof}
	Let $N\subset \N$ be a set with unbounded gaps. We will construct an SIP set contained in $\N\setminus N$. Since predictive sets are $\SIP$, this will complete the proof.
	
	Choose $s_1\in \N\setminus N$ and fix $S_1:=\{s_1\}$. Since $N$ does not have bounded gaps, $\N\setminus N$ has infinitely many intervals of length greater than $2 s_1$. Thus there exists $s_2>s_1$ such that $s_2, s_2-s_1, s_2+s_1\in \N\setminus N$ and hence
	$S_2:=\SIPp(s_1, s_2)\subset \N\setminus N$. Again, there exists $s_3>s_1+ s_2$ such that 
	$$\{s_3\}, s_3-S_2, s_3+S_2\subset \N\setminus N$$
	giving us that $S_3:=\SIPp(s_1, s_2, s_3)\subset \N\setminus N$. Continuing in this manner we get a sequence of natural numbers $s_i; i \in \N$ such that
	$\SIPp(s_1, s_2, \ldots )\subset \N\setminus N$.
\end{proof}

To prove Theorem \ref{thm:SIpstar} it will be simpler to pass onto infinite-valued stationary processes. We will make the assumption that our processes take values in a separable Borel space $R$. Note that now by the entropy of a process $X_\Z$ we mean the Kolmogorov entropy of the measure preserving transformation defined by the process. This entropy is defined to be the supremum of the average Shannon entropy of the finite-valued processes $X^\alpha_\Z$  that are obtained by composing $X_\Z$ with finite-valued measurable mappings $\pi:R\to C_\alpha$ with $C_\alpha$ being a finite set. This can be finite even if the Shannon entropy of $X_0$ is infinite. For example consider the constant process $X_\Z$ with state space $\N$ for any probability measure on $\N$ with infinite Shannon entropy.

\begin{proposition}\label{Proposition:infinite_valued_predictive}
	$P$ is a predictive set if and only if for all zero entropy stationary processes (possibly infinite-valued) $X_\Z$, $X_0$ is measurable with respect to $X_P$.
\end{proposition}
\begin{proof}Let $P$ be a predictive set and $X_\Z$ be a zero entropy stationary process taking values in the separable Borel space $R$. Let $\mathcal B_i; i \in \N$ be a sequence of partitions which together generate the sigma algebra of $R$. Consider the factor $\pi_i: R^\Z\to (\mathcal B_i)^\Z$ given by $(\pi_i(X_\Z))_n:= B$ if $X_n\in B\in \mathcal B_i$. But $\pi_i(X_\Z)$ is a finite-valued zero entropy process. Thus $(\pi_i(X_\Z))_0$ is measurable with respect to $(\pi_i(X_\Z))_P$ and hence with respect to $X_P$. Since $\mathcal B_i; i \in \N$ together generate the sigma algebra of $R$ we have that
$X_0$ is measurable with respect to $X_P$.
\end{proof}

We will now need to recall some properties of Gaussian processes for our purposes. A \emph{Gaussian process} is a complex-valued stationary process $X_\Z$ such that $X_{n}, X_{n+1}, \ldots, X_m$ are jointly Gaussian for all $n<m$. We shall assume throughout that $X_n$ are zero mean for all $n \in \Z$. Since jointly distributed Gaussian random variables are determined by their covariance, Gaussian processes are determined by their autocorrelation sequence: $\mathbb E(X_0\overline{X_n}); n\in \N$ and thereby the corresponding spectral measure $\mu$ on $\mathbb T$ for which the Fourier coefficients are given by 
$$\hat{\mu}(n)= \mathbb{E}(X_0\overline{X_n}).$$
In the same vein every finite non-negative measure $\mu$ on $\mathbb{T}$ uniquely determines a Gaussian process with spectral measure $\mu$.
Thus the properties of the Gaussian process can be determined by its spectral measure and vice versa. We state few properties of Gaussian processes $X_\Z$:
\begin{enumerate}[(i)]
	\item A Gaussian process is ergodic if and only if the spectral measure is continuous if and only if it is weak mixing.  \cite[{Page 191}]{MR832433} 
	\item A Gaussian process has zero entropy if and only if the spectral measure is singular. \cite{MR1231420}. 
	\item $\mathbb E(X_0~|~ X_\N)=X_0$ if and only if the spectral measure does not have an integrable log-density or is singular. This follows from the well-known theorems of Szeg\"{o} \cite{szego} and Verblunsky \cite{MR1575824} (also \cite[Theorem 3]{MR2956573} and \cite[Chapter 2]{MR2743058}). We will revisit this theorem later in Section \ref{Section:Comments Gaussian}.
\end{enumerate}
The proof of Theorem \ref{thm:SIpstar} will follow from the following stronger proposition.
\begin{proposition}\label{Proposition: Independence_of_SIP}
	Let $S$ be an SIP set. Then there exist a zero entropy weakly mixing Gaussian process $X_\Z$ such that $X_0$ is independent of $X_{\N\setminus S}$.
\end{proposition}
\begin{remark}
	In Section \ref{Section: Notation} we constructed a zero entropy process where $X_{k\N+r}$ is independent of $X_0$ for $r\notin k\Z$ but the process was not weak mixing. This construction will give an example which is weak mixing.
\end{remark}

\begin{proof}
Fix a sequence of natural numbers $s_i; i \in \N$. Consider a sequence of mutually disjoint sets of indices $I_n$ such that
$$r_n:=\sum_{i\in I_n}s_i$$
satisfies $r_{n+1}>3r_n$. Clearly $\SIPp(r_1, r_2, \ldots)\subset \SIPp(s_1, s_2, \ldots)$.

Let $\mu_l$ denote the Lebesgue measure on $\mathbb T$. The Riesz product $f_n:\mathbb T\to [0, \infty)$ is given by
$$f_n(x):=\prod_{m=1}^n(1+ \cos(2\pi r_m x)).$$
The corresponding probability measures with densities $f_n$ converge weak-star to a singular continuous probability measure $\mu$ on $\mathbb{T}$ \cite[Pages 208-211]{MR1963498}. Further, the Fourier coefficients of $f_n$ are supported on $\SIPp(r_1, r_2, \ldots, r_n)$. Therefore the Fourier coefficients of $\mu$ are supported on $\SIPp(r_1, r_2,\ldots)$.

Let $X_\Z$ be the Gaussian process with spectral measure $\mu$. Then $X_\Z$ is weak mixing, has zero entropy and further
$$\mathbb E(X_0\overline{X_n})=\int_\mathbb{T} e^{-2\pi i nx}d\mu(x)$$
is non-zero if and only if $n\in \SIPp(r_1, r_2, \ldots)$.
Thus we have that $X_0$ is independent of $X_n$ if and only if $n\notin \SIPp(r_1, r_2, \ldots)$. 
\end{proof}

\begin{proof}[Proof of Theorem \ref{thm:SIpstar}]
Let $P$ be a predictive set and $Q$ be an SIP. By Proposition \ref{Proposition: Independence_of_SIP} there exists a zero entropy Gaussian process $X_\Z$ such that $X_0$ is independent of $X_{\N\setminus Q}$. By Proposition \ref{Proposition:infinite_valued_predictive} we have that $X_0$ is measurable with respect to $X_P$. Thus $P\cap Q\neq \emptyset$.
\end{proof}

\begin{remark}
The use of Proposition \ref{Proposition:infinite_valued_predictive} above can be circumvented. More directly, let $Y_\Z$ be the process given by $Y_i:=sign(X_i)$. The process $Y_\Z$  is a zero entropy weakly mixing two-valued process for which $Y_0$ is independent of $Y_{\N\setminus Q}$. Thus if $P$ and $Q$ are disjoint then $P$ is not predictive even for two-valued processes.
\end{remark}

\begin{question}\label{question:sip_predictive}
	Are all $\SIP$ sets predictive? 
\end{question}

In \cite{MR3499490} Glasner and Akin constructed two $\SIP$ sets which do not intersect each other. In fact they showed that for any irrational $\alpha\in \mathbb T$ and $0<\epsilon<1/2$, the set $\{n~:~n\alpha\in(0,\epsilon)\}$ is $\SIP$. For the reader's convenience, we outline their proof.

 Given an infinite sequence of natural numbers $s_i; i \in \N$ there exist a strictly increasing sequence $r_i; i \in \N$ such that 
$$SIP^+(r_1, r_2,...)\subset SIP^+(s_1, s_2,...).$$
 By compactness of $\mathbb T$ there exists $i<j<j'<k$ such that $r_k-r_{j'}>r_j-r_i$,
$$ (r_j-r_i)\alpha\in (-\epsilon, \epsilon)\text{ and }|(r_k-r_{j'})\alpha|< |(r_j-r_i)\alpha|.$$  
Now if neither $(r_k-r_{j'})\alpha$ nor $(r_j -r_{i})\alpha$ are  elements of $(0, \epsilon)$ then they must be elements of $(-\epsilon, 0)$ and it follows that $((r_k-r_{j'})-(r_j-r_i))\alpha \in (0, \epsilon)$. This shows that $\{n~:~n\alpha\in (0, \epsilon)\}$ is $\SIP$. 

\begin{question}\label{question:filter or sip}
Let $\alpha\in \mathbb T$ be irrational and $1/2>\epsilon>0$. Is the set $\{n~:~n\alpha\in (0, \epsilon)\}$ predictive?	
\end{question}

Answering this question either way will have interesting consequences. If the answer is an affirmative then this would show that predictive sets don't form a filter answering Question \ref{Question:Filter}. If the answer is negative it would give us an example of an $\SIP$ set which is not predictive answering Question \ref{question:sip_predictive}.

Let $\Pred$ denote the set of predictive sets and $$\Pred^*:=\{Q\subset \N~:~Q\cap P\neq \emptyset\text{ for all }P\in \Pred\}.$$
\begin{question}
Let $Q\in \Pred^*$. Does there exists a zero entropy process $X_\Z$ such that $X_0$ is independent of $X_i$ for $i\in \N\setminus Q$?
\end{question}

To arrive at the necessary conditions for a set to be predictive we studied weak mixing Gaussian processes. This naturally gives rise to the following questions.
\begin{question} 		
	\begin{enumerate}[(i)]
		\item Is it true that $P$ is predictive if and only if for all zero entropy Gaussian processes $X_\Z$, $X_0$ is measurable with respect to $X_P$?
		\item Suppose $P\subset \N$ is such that for all ergodic/weak mixing zero entropy processes  $X_\Z$, $X_0$ is measurable with respect to $X_P$. Is $P$ a predictive set?
	\end{enumerate}
\end{question}

A topological dynamical system $(X, T)$ is called \emph{minimal} if the only closed invariant subsets of $X$ are the empty set and the entire set $X$. If $G$ is a compact group and $\alpha \in G$ then the rotation $(G, R_\alpha)$ is minimal if and only if the set $\{\alpha^n~:~n \in \N\}$ is dense. As a consequence of Corollary \ref{corollary:predictive_rotation} we have that if $x\in U \subset G$ where $U$ is an open set then $N(x, U)$ is a predictive set. Further by Corollary \ref{cor: bounded gaps predictive} we have that any predictive set must has bounded gaps. It is easy to see that if $(X, T)$ is a minimal system, $ U \subset X$ and $x\in X$ where $U$ is an open set then the set $N(x, U)$ has bounded gaps. It is now a natural question whether all such visit times $N(x, U)$ are predictive. We will now use the fact that predictive sets are $\SIP$ to see that this is not true.
\begin{proposition}\label{Proposition:minimal_not_predictive}
There exists a minimal system $(X, T)$, open set $U \subset X$ and $x\in U$  such that $N(x, U)$ is not a predictive set.
	\end{proposition}

A set $A\subset \Z$ is called \emph{small} if orbit closure of the characteristic function $1_A\in \{0,1\}^\Z$ contains a unique minimal subshift $\{0\}^\Z$. Equivalently, a set $A\subset \Z$ is \emph{small} if and only if for all $n\in \N$ the set 
$$\{i~:~[i, i+n]\cap S=\emptyset\}$$
has bounded gaps. We will need the following result of Glasner and Weiss \cite{MR710239} (see also \cite[Section 10]{glasner2019bernoulli} for a new proof which extends to all countable groups). 

\begin{proposition}\label{Proposition: Glasner Weiss}
	If a set $A\subset \Z$ is small then there exists a minimal subshift $X\subset \{0,1\}^\Z$ such that for all $C\subset A$, there exists $x\in X$ such that $x|_A=1_C$.
	\end{proposition}

The converse is also true and has been proved in \cite{MR710239}.

\begin{proof}[Proof of Proposition \ref{Proposition:minimal_not_predictive}]
Since predictive sets are $\SIP$ it is sufficient to find an SIP $S\subset \N$ and a point $x\in \{0,1\}^\Z$ contained in a minimal shift such that $x_0=1$ and $x_i=0$ for all $i \in S$. This would imply that $N(x,[1])$ is disjoint from $S$ and imply that it is not predictive. By the previous proposition (Proposition \ref{Proposition: Glasner Weiss}) it is sufficient to find an SIP $S\subset \N$ such that $S\cup \{0\}$ is small. 

Let $s_i=5^i$ for $i \in \N$. Then $S:=\SIPp(s_1,s_2, \ldots)$ is small: Numbers in $S$ are contained in those whose base $5$ expansion consists of $0, 1, 3$ and $4$. 
Thus for all $n\in \N$ we have that 
$$\{i \in \Z~:~[i, i+n]\cap S = \emptyset\}$$
cannot have a gap bigger than $n+ 5^{m+1} $ where $m$ is the least number such that $n<5^m$.
	\end{proof}

It is still an interesting question whether the visit times of specific minimal systems are predictive.

\begin{question}
	Let $(x_i)_{i\in \N\cup \{0\}} = 1,0,0,1,0,1,1,0, \ldots $ be the Thue-Morse sequence. Is $\{i\in \N~:~ x_i=1\}$ predictive?
\end{question}

We end this section with some indication why we think the answer to Question \ref{Question:positive definite} is an affirmative. Given a set $A\subset \N$, we denote its \emph{difference set} by 
 $$\Delta(A):=\{|a-b|~:~a, b\text{ are distinct elements of }A\}.$$  A $\Delta_r$-set is a difference set of a set $A$ where $|A|=r$.
 A set $Q\subset \N$ is called $\Delta_r^\star$ if for all $\Delta_r$-sets $A$, $A\cap Q\neq \emptyset$. 
 
\begin{remark}\label{remark:about deltastar}
To illustrate the definition let us consider small values of $r$. 
\begin{enumerate}[(i)]
	\item $\Delta_2$-sets are precisely all the singletons of $\N$ and consequently $\N$ is the only $\Delta_2^\star$-set. 
	\item Suppose $x<y<z$ are natural numbers. If we set $a:=y-x$ and $b=z-y$, we see that $\Delta(\{x,y,z\})=\{a,b,a+b\}$.
	One can quickly deduce that $\Delta_3$-sets are precisely all the sets of the type $\{a, b, a+b\}$ for some $a, b \in \N$. Thus $\Delta_3^\star$-sets are precisely those sets $P$ for which if $a,b\in \N\setminus P$ then $a+b\in P$, that is, sets $P$ for which $(\N\setminus P)+(\N\setminus P)$ is disjoint from $(\N\setminus P)$.
	\item In a similar fashion as above, we can deduce that $\Delta_4$-sets are precisely the sets of the form
	$$\{a,b,c,a+b, b+c, a+b+c\}$$
	for $a, b, c\in \N$. Since $a=(a+b+c)-(b+c)=(a+b)-b$, if $P\subset \N$ is such that there do not exist $x<y\leq z<w\in \N\setminus P$ such that $y-x=w-z$ then $P$ is a $\Delta_4^\star$. In particular, if $\N\setminus P=\{n_i~:~i \in \N\}$ where $n_i$ is a strictly increasing sequence for which $n_{i+1}-n_i$ is also strictly increasing then $P$ is $\Delta_4^\star$.
\end{enumerate}

 \end{remark}

Since SIP sets contain difference sets of all sizes we have that $\Delta_r^\star$-sets are $\SIP$. Thus $\Delta_r^\star$-sets have bounded gaps. The following proposition is well-known \cite[Page 4]{MR1774423}.
 
\begin{proposition}\label{prop:bounded gaps of positive definite sequecnes}
	Fix $\epsilon>0$ and a real valued positive definite sequence $a_n; n\in \Z$. The set 
	$$\{n\in \N~:~a_n >-\epsilon\}$$
	is $\Delta_r^\star$ for some $r\in \N$. In particular, it has bounded gaps. 
\end{proposition}

This shows evidence that the set $\{n\in \N~:~a_n >-\epsilon\}$ is a predictive set. We had raised this issue earlier in Question \ref{Question:positive definite}. We will use the following fact to prove the proposition.

\begin{lemma}\label{lemma: Hilbert space inner}
	Let $\epsilon>0$ and $H$ be a Hilbert space with inner product $\langle\cdot, \cdot \rangle$. If $v_1, v_2, \ldots, v_r\in H$ are unit vectors such that for all distinct $i, j$,
	$\langle v_i, v_j\rangle\leq -\epsilon$ then $r\leq 1+ 1/\epsilon$. 
\end{lemma}
We prove this for the sake of completeness.
\begin{proof}
Let $V:=v_1+v_2+\ldots+v_r$. Then we have
$$0\leq \langle V, V \rangle = \sum_{i=1}^r\langle v_i, v_i\rangle + \sum_{i\neq j} \langle v_i, v_j\rangle \leq r -(r^2-r)\epsilon.$$
Thus $r\leq 1+ 1/\epsilon$.
\end{proof}

\begin{proof}[Proof of Proposition \ref{prop:bounded gaps of positive definite sequecnes}]
Let $H$ be a Hilbert space with inner product $\langle\cdot, \cdot \rangle$, $U$ be a unitary operator and $v\in \H$ such that 
$a_n:=\langle U^n v, v \rangle$. Let $A:=\{i_1, i_2, \ldots, i_r\}\subset \N$ where $r$ is an integer such that $r>1+ \frac{ \langle v, v\rangle}{\epsilon}$. It follows from Lemma \ref{lemma: Hilbert space inner} that there must be $1\leq j <k \leq r$ such that 
$$\langle U^k v, U^j v \rangle > -\epsilon.$$
It follows that $a_{k-j}=\langle  U^{k-j} v,  v \rangle >-\epsilon$. 
\end{proof}

\section{Linear predictivity}\label{Section:Comments Gaussian}

In this section we will talk about a much stronger form of prediction: Let $(\Omega, \mathcal B,\nu)$ be a probability space and $X, Y_i\in L^2(\Omega);i \in \N$ be complex-valued random variables. We will say that $X$ can be \emph{linearly predicted }using $Y_\N$ if $X$ is in the closed Hilbert space generated by $Y_\N$. If $H(X~|~Y_\N)=0$ we will say that $X$ can be \emph{predicted} using $Y_\N$.

We now recall Szeg\"o's \cite{szego} and Verblunsky's theorem \cite{MR1575824} (abbreviated as S-V theorem from here on). For a proof in the form that we state the results do consider \cite[Page 122]{MR2656971} and \cite{MR350291}. In the following let $\mu_l$ denote the Lebesgue measure on $\mathbb T$.

\begin{theorem}
	\begin{enumerate}[(i)]
		\item Let $\mu$ be an absolutely continuous probability measure on $\mathbb T$ with density $f$. Then
		$$ \exp(\int_{\mathbb{T}}\log(f)d\mu_l)= \inf_{p}\int_{\mathbb{T}}|1+p|^2d\mu$$
		where $p$ ranges over polynomials $\sum_{j=1}^n a_j e^{2\pi i j x}$ where $n\in \N$ and $a_j\in \mathbb C$. The left hand side is taken to be zero when $\log(f)$ is not integrable.
		\item If $\mu$ is a singular probability measure on $\mathbb T$ then 
		$$\inf_{p}\int_{\mathbb{T}}|1+p|^2d\mu=0$$
		where $p$ ranges over the same set of polynomials as above.
	\end{enumerate}
 \end{theorem}

Let us see how this relates to linear predictivity. Let $(Y, \mu, S)$ be a probability preserving transformation and $g\in L^2(Y)$ and spectral measure $\mu_g$. Let $H$ be the closed linear span of $T^i(g);i\in \N$. We know that 
$$\| 1+ \sum_{j=1}^n a_j e^{2\pi i j}\|_{L^2(\mu_g)}=\| g+ \sum_{j=1}^n a_j T^j(g)\|_{L^2(\mu)}.$$
Thus S-V theorem implies that $g$ belongs to $H$ if and only if either $\mu_g$ is singular or it has a non-integrable log density.

In particular if $(Y, \mu, S)$ has zero entropy and has a Lebesgue component in its spectrum, it would give an example of a process $X_\Z$ for which $X_0$ can be predicted by $X_\N$ but not linearly. Starting with constructions by Newton and Parry in \cite{MR206209} several such examples have been constructed. We will mention one which is particularly easy to check and is based on the example by Flaminio which is $k$-fold independent, mixing but has zero entropy \cite{MR1154245}.

\begin{example}
	Let $\mu$ denote the Lebesgue measure on $\mathbb T^2$ and $T:\mathbb T^2\to \mathbb T^2$ be given by $T(x,y):=(x+y, y)$. The ergodic components of $(\mathbb T^2, \mu, T)$ are isomorphic to rotations of the circle and hence the ppt has zero entropy. Let $X_\Z$ on $(\mathbb T^2, \mu)$ be the process given by $X_i(x,y):=2\:1_{[0,1/2]\times \mathbb T}(T^i(x,y))-1$. Then $\mathbb E(X_i X_j)=1_{i=j}$ for $i\neq j$. Since the $X_i$'s take two values $1$ and $-1$ and have mean-zero, it follows that the process is $2$-fold independent. In fact $ (\mathbb T^2, \mu, T)$ has countable Lebesgue spectrum in the ortho-complement of the space of invariant functions: Let $f_{n,k}(x, y):=e^{2\pi i (nx+k y)}$ for $n, k\in \Z$. Then the collection 
	$$T^r(f_{n,k})=f_{n,k+nr}; r\in \Z, n\in \Z\setminus \{0\}, k\in \{1,2, \ldots, |n|\} $$	forms an orthonormal basis in the ortho-complement of space of invariant functions. 
	 
	 This gives us an example where $X_0$ can be predicted using $X_\N$ but cannot be predicted linearly.  Though this example is not ergodic, by a slight modification (as in \cite{MR1154245}) one can construct an ergodic example: The transformation involved is a skew product over an irrational rotation by $\alpha \in \mathbb{T}$ denoted by $S: \mathbb T^2\to \mathbb{T}^2$ given by $S(x,y):=(x+y,y+\alpha)$ while the process $Y_\Z$ on $(\mathbb T^2, \mu, S)$ is given by $Y_i(x,y):=2\:1_{[0,1/2]\times \mathbb T}(T^i(x,y))-1$.
	 
	Mixing examples can be constructed using horocycle flows \cite{MR0488168,Kushnirenko_1967}, factors of Gaussian processes \cite{MR206209} or following Flaminio's example \cite{MR1154245}.
\end{example}

 Before going further into linear prediction we will first recall some aspects of Gaussian processes. 
 
 \begin{proposition}\label{prop:Gaussian_predict}
 	Let $X_\Z$ be a Gaussian process and $P\subset \Z$. Then $\mathbb{E}(X_0~|~X_P)$ is in the closed linear span of $X_i; i\in P$. Thus $X_0$ can be predicted by $X_P$ if and only if it can be predicted linearly.
 \end{proposition}
 
 \begin{proof}
Let $X_\Z$ be a Gaussian process on the probability space $(\Omega, \mathcal B, \mu)$. Let $H\subset L^2(\Omega)$ be the Hilbert space generated by $X_i; i \in P$ and consider the decomposition $X_0= X_0'+X_0''$ where $X_0'\in H$ and $X_0''\in H^\perp$. Then $(X_0', X_P)$  and $(X_0'', X_P)$ are also jointly Gaussian. On the other hand, orthogonality is the same as independence for Gaussian random variables. Thus
$$\mathbb{E}(X_0''~|~X_P)=\mathbb{E}(X_0'')=0$$
and thus
$$\mathbb{E}(X_0~|~X_P)=\mathbb{E}(X_0'~|~X_P)+\mathbb{E}(X_0''~|~X_P)=X_0'\in H.$$
Now if $X_0$ can be predicted by $X_P$ then 
$$X_0=\mathbb{E}(X_0~|~X_P)\in H$$
implying that $X_0$ can be predicted linearly. The converse is obvious.
 \end{proof}
 
Look at \cite{vershik1964some} to see how a generalisation of such a property can characterise ergodic Gaussian processes. We can use this to prove linear predictivity in processes with singular spectral measure.
 
 \begin{proposition} Let $X_\Z$ be a real-valued (possibly infinite-valued) process on the probability space $(\Omega, \mathcal B, \mu)$ such that $X_0\in L^2(\mu)$. If it has singular spectral measure and $P$ is a predictive set then $X_P$ can linearly predict $X_0$.
 \end{proposition}
 
 \begin{proof} Enumerate $P:=\{p_1, p_2, \ldots\}$
 	Let $\mu$ be the spectral measure for $X_\Z$. Let $Y_\Z$ be a Gaussian process with the same spectral measure $\mu$. Since $\mu$ is singular, $Y_\Z$ has zero entropy and 
 	$Y_0$ can be predicted (Proposition \ref{Proposition:infinite_valued_predictive}) and hence linearly predicted (Proposition \ref{prop:Gaussian_predict}) by $Y_P$. This implies that
 		$$\inf_{q}\|1+q\|_{L^2(\mu)}=0$$
 		where $q$ ranges over polynomials $\sum_{j=1}^n a_je^{2\pi i p_j}$.
 		But since $\mu$ is the spectral measure for $X_\Z$ we have that
 		$$\inf_{q}\|X_0+q\|=0$$
 	where $q$ ranges over sums of the form $\sum_{j=1}^na_j X_{p_j}$. Thus $X_P$ can linearly predict $X_0$. 
 \end{proof}
 
The same proof does not go through when the spectral measure has a density whose logarithm is non-integrable. Though we can apply S-V theorem and conclude that $X_0$ can be linearly predicted by $X_\N$, we cannot apply our results on predictive sets on the Gaussian process with the same spectral measure because it has infinite entropy. The following question arises.

\begin{question}\label{question: log_density_integrate_predict}
	Let $P$ be a predictive set and $X_\Z$ be a zero entropy real-valued process on the probability space $(\Omega, \mathcal B, \mu)$ such that $X_0\in L^2(\mu)$ and its spectral measure has a density whose logarithm is not integrable. Does this imply that $X_0$ can be linearly predicted by $X_P$ for predictive sets $P$?
\end{question}

To further exemplify the issue we use the following proposition which involves a Gaussian process with infinite entropy.
 
\begin{proposition}
For all $k\in \N$ there exists a Gaussian process $X_\Z$ for which $X_0$ can be linearly predicted by $X_{\N}$ but cannot be predicted by $X_{k\N}$.
\end{proposition}  

\begin{proof}
Consider the measure $\mu$ on $\mathbb T$ given by the Lebesgue measure supported on $[-1/2k,1/2k]$ and let $X_\Z$ be the corresponding Gaussian process. By S-V theorem we know that $X_0$ can be linearly predicted by $X_\N$. However
$$\mathbb E (X_0\overline{X_n})=\int_{\mathbb T}e^{-2\pi i nt} d\mu(t)$$
is $0$ if $n$ is a multiple of $k$. This shows that $X_0$ is independent of $X_{k\N}$. 
\end{proof}

\begin{question}
 Let $P\subset \N$ be a set such that if for a Gaussian process $\mathbb{E}(X_0~|~X_{\N})=X_0$ then $\mathbb{E}(X_0~|~X_{P})=X_0$. Is $\N\setminus P$ necessarily finite?
\end{question}

As a remark, we leave with  a proposition which we expand on further in the next section.  In the following and later given a measure $\mu$ and a set $H\subset L^2(\mu)$ we will denote by $\overline H$ the closure of $H$ in the $L^2$ metric and by $\text{Span}(H)$ we will denote the linear span of $H$. 

\begin{proposition}\label{prop: singular measure prob riesz}
	Let $P$ be a predictive set. Then for any singular finite non-negative measure $\mu$, 
	$$1\in \overline{\text{Span}\{e^{2\pi i p x}~:~p \in P\}}.$$
	Thus there exists $p \in P$  such that $\hat\mu(p)\neq 0.$
	\end{proposition}

\begin{proof}
Let $X_\Z$ be the Gaussian process with spectral measure $\mu$ and zero mean. Since $\mu$ is singular we have that $X_\Z$ has zero entropy. By Proposition \ref{prop:Gaussian_predict}, $X_0$ is in the closed linear span of $X_p; p \in P$. Since $\mu$ is the spectral measure for $X_\Z$ 
$$\inf_{q'}\|1+q'\|_{L^2(\mu)}=\inf_{q}\|X_0+q\|=0$$
where $q'$ ranges over polynomials $\sum_{j=1}^k a_je^{2\pi i p_j x}$, $q$ ranges over sums of the form $\sum_{j=1}^k a_j X_{p_j}$ and $p_j\in P$ for $1\leq j \leq k$. Thus 
	$$1\in \overline{\text{Span}\{e^{2\pi i p x}~:~p \in P\}}$$
and $\hat\mu(p)\neq 0$ for some $p\in P$. 
\end{proof}

\section{Totally Predictive Sets and Riesz Sets}\label{section:Totally predictive}

Up until now we were involved in prediction of a single random variable $X_0$. In this section we will be interested in totally predictive sets; sets which can predict the entire future. Furthermore once we predict the entire future, reversing time, we can fill in the past and recover the entire process. The subject becomes more intriguing because of its connections with Riesz sets and parallels in constructions of totally predictive and Riesz sets. Let us begin with a few definitions.

\begin{definition}
	A set $P\subset \N$ is \emph{totally predictive} if for all zero entropy processes, $X_\Z$, $X_n$ is measurable with respect to $X_{P}$ for all $n\in \Z$.
\end{definition}

Clearly $\N$ is a totally predictive set: Given that $X_0$ is measurable with respect to $X_\N$, $X_{-1}$ is measurable with respect to $X_{\N\cup \{0\}}$ and so on; it follows automatically that $X_{-1}, X_{-2}, \ldots $ are also measurable with respect to $X_\N$. We begin by expanding further on Proposition \ref{prop: singular measure prob riesz} and making a connection between totally predictive sets and harmonic analysis. From here onwards, by a measure we mean a finite complex-valued measure on the circle. Given a measure $\mu$, by Hahn-Jordan decomposition theorem we can write 
$$\mu=\mu^+_{\text{real}}-\mu^-_{\text{real}}+i (\mu^+_{\text{imag}}-\mu^-_{\text{imag}})$$ where $\mu^+_{\text{real}}$, $\mu^-_{\text{real}}$ and $\mu^+_{\text{imag}}$, $\mu^-_{\text{imag}}$ are uniquely defined pairs of mutually singular non-negative finite measures. The \emph{absolute value} of $\mu$ is denote by
$$|\mu|:=\mu^+_{\text{real}}+\mu^-_{\text{real}}+ \mu^+_{\text{imag}}+\mu^-_{\text{imag}}.$$
It follows that there exists a measurable function $f: \mathbb T\to \R$ such that $d \mu= f d|\mu|$ and there exists constants $c, C>0$ such that $c<|f|<C$. This implies that $\mu$ is singular with respect to the Lebesgue measure if and only if $|\mu|$ is as well and $f\in L^2(|\mu|)$.
\begin{proposition}\label{proposition: total predictive singular}
Let $P\subset \N$  be a totally predictive. Then for any finite singular measure $\mu$ on the circle 
$$L^2(|\mu|)=\overline{\text{Span}\{e^{2\pi i p x}~:~p \in P\}}.$$  
Further there exists $p \in P$ such that $\hat\mu(p)\neq 0$ for some $p \in P$. 
\end{proposition}

\begin{proof}
Let $\mu$ be a singular measure on the circle, $d\mu=fd|\mu|$ and $X_\Z$ be the Gaussian process with spectral measure $|\mu|$. Since $|\mu|$ is singular we have that $X_\Z$ has zero entropy and $X_n$ is in the closed linear span of $X_p; p\in P$ for all $n \in \Z$ and
$$\inf_{q'}\|e^{2\pi i n x}+q'\|_{L^2(|\mu|)}=\inf_{q}\|X_n+q\|=0$$
where $q'$ ranges over polynomials $\sum_{j=1}^k a_je^{2\pi i p_j x}$,  $q$ ranges over sums of the form $\sum_{j=1}^k a_j X_{p_j}$ and $p_j\in P$ for $1\leq j \leq k$. Thus 
$$f\in L^2(|\mu|)=\overline{\text{Span}\{e^{2\pi i n x}~:~n \in \Z\}}=\overline{\text{Span}\{e^{2\pi i p x}~:~p \in P\}}\text{ and}$$  
$$\hat \mu(p)=\int_{\mathbb T}e^{-2\pi i p x}d\mu=\int_{\mathbb T}e^{-2\pi i p x}fd|\mu|\neq 0$$
for some $p \in P$.
\end{proof}
Given a finite measure $\mu$ on the circle we write the support of the Fourier coefficients of $\mu$ by
$$\supp(\hat\mu):=\{n \in \Z~:~\hat{\mu}(n)\neq 0\}.$$
Rephrasing Proposition \ref{proposition: total predictive singular} we get the following corollary.
\begin{corollary}\label{corollary:singular not outside predictive}
	If $P$ is totally predictive and $\mu$ is a finite measure such that $\supp(\hat{\mu})\subset \Z\setminus P$ then $\mu$ is not singular.
\end{corollary}

\begin{remark}
The assumption that $P$ is a predictive set can be weakened, at least, formally. It is sufficient for the conclusion of Proposition \ref{proposition: total predictive singular} that $P$ is totally predictive for all zero entropy Gaussian processes. However we do not know if there are sets $P$ which are totally predictive for all zero entropy Gaussian processes but not for all zero entropy processes. 

\end{remark}

Keeping the conclusions of Corollary \ref{corollary:singular not outside predictive} in mind, let us recall the following definition introduced by Meyer \cite{MR227697}.

\begin{definition}
A \emph{Riesz set} is a set $Q\subset \Z$ such that all measures $\mu$, satisfying $\supp(\hat\mu)\subset Q$, are absolutely continuous.
\end{definition}
The following questions immediately arise.
\begin{question}\label{question:are totally predictive riesz}
If $P\subset \N$ is totally predictive then is the set $\Z\setminus P$ a Riesz set? If $Q\subset \N$ is a set whose closure in the Bohr topology is a Riesz set then is the set $\N\setminus Q$ a totally predictive set?
\end{question}

We give a partial answer. For this we will need to introduce \emph{the Bohr topology}: This is the topology induced on $\Z$ from the \emph{Bohr compactification of $\Z$}. This in-turn is the Pontryagin dual of $\mathbb T$ endowed with the discrete topology. In more concrete terms, it is the topology on $\Z$ generated by a basis consisting of sets $Q\subset \Z$ of the following type: There exists $\alpha\in \mathbb T^d$ and an open set $U \subset \mathbb T^d$ such that
$$Q=\{i \in \Z~:~i\alpha\in U\}.$$
It follows that the restriction of the topology to $\N$ is generated by the visit-times of toral rotations. 

\begin{theorem}\label{thm:totally predictive open is Riesz}
	Let $P$ be a totally predictive set which is open in the restriction of the Bohr topology to $\N$. Then $\Z\setminus P$ is a Riesz set.
\end{theorem}

Given a finite measure $\mu$ on the circle we denote its singular component by $\mu_s$. To prove the proposition we will need the following facts: Firstly we identify the support of the Fourier coefficients of the singular components of finite measures.

\begin{proposition}(\cite[Theorem 5]{MR227697})\label{prop: singular component support}
If $\mu$ is a finite measure on $\mathbb T$ then $\supp(\hat{\mu}_s)$ is contained in the Bohr closure of $\supp(\hat{\mu})$.
\end{proposition}

The second fact that we need is that the union of a Riesz set which is closed and a Riesz set is still a Riesz set.

\begin{proposition}(\cite[Theorem 2]{MR227697}) \label{prop: closed riesz} Let $P, Q\subset \Z$ be Riesz sets and $P$ be closed in the Bohr topology. Then $P\cup Q$ is a Riesz set.
\end{proposition}

Finally we will recall a famous theorem of Frigyes and Marcel Riesz which inspired the study of Riesz sets. A short proof can be found in \cite{oksendal1971short}.
\begin{theorem}\label{theorem:Riesz}
The set $\Z\setminus \N$ is a Riesz set.
\end{theorem}
We remark that the terminology of Riesz set was introduced much later by Meyer in \cite{MR227697}.
\begin{proof}[Proof of Theorem \ref{thm:totally predictive open is Riesz}]
	Let $\mu$ be a finite measure on $\mathbb{T} $ such that $\supp(\hat\mu)\subset \Z\setminus P$. From Proposition \ref{prop: singular component support} it follows that $\supp(\hat{\mu}_s)$ is contained in the closure of $\Z\setminus P$. Since $P$ is open in the restriction of the Bohr topology to $\N$ it follows that
	$$\supp(\hat{\mu}_s)\subset \Z\setminus P.$$
	Since $P$ is totally predictive by Corollary \ref{corollary:singular not outside predictive} we have that $\mu_s=0$. Thus $\mu$ is absolutely continuous and the closure of $\N\setminus P$ is a Riesz set.
	It follows thus from Theorem \ref{theorem:Riesz} and Proposition \ref{prop: closed riesz} that $(\N\setminus P)\cup (\Z\setminus \N)=\Z\setminus P$ is a Riesz set.
	\end{proof}

While we do not know the answer to Question \ref{question:are totally predictive riesz} we have many examples of sets which are both Riesz and complements of totally predictive sets.

We will mention a few of them. 

A \emph{lacunary set $\Lambda\subset \N$ with a constant $\rho>1$} is a set $\Lambda=\{\lambda_1, \lambda_2, \lambda_3, \ldots, \}$ such that $\frac{\lambda_{i+1}}{\lambda_i}>\rho$. In \cite[Page 226]{MR0116177}, Rudin proved that the union of lacunary sets and $\Z\setminus \N$ are Riesz sets (the terminology was introduced later by Meyer).

Given a sequence of integers $n_i$ we write 
$$\text{IP}(n_1, n_2, \ldots):=\left\{\sum_{i=\N}\epsilon_i n_i \text{ where }\epsilon_i\in \{0,1\}\right\}.$$
In \cite{MR227697}, Meyer proved that $Q\cup (-\N)$ are Riesz where $Q$ are the sets:

\begin{enumerate}[(i)]
	\item Finite union of lacunary sets (simplifying Rudin's argument from \cite{MR0116177}).
	\item $\{n^2~:~n\in \N\}$.
	\item The set of prime numbers.
	\item Sets $Q\subset \N$ for which there exist sequences of natural numbers $l_k, n_k$ where $n_k$ is an increasing sequence, $n_k-l_k\to \infty$ and $Q\setminus \cup_{t\in \N}[tn_k, tn_k+l_k]$ is finite for all $k \in \N$. 
	\item $\text{IP}(n_1, n_2, \ldots)$ where $n_i\in \N$ is a sequence such that $\frac{n_{i+1}}{n_i}\in \N$ is greater than or equal to $3$.
\end{enumerate}

We will prove that the complement of these sets are totally predictive using essentially the same ideas as in \cite{MR227697}. Our ideas also imply that for these sets $Q$ there exists a finite set $S$ such that $\Z\setminus (Q\cup S)$ is open in the Bohr topology. It follows from Theorem \ref{thm:totally predictive open is Riesz} that they are Riesz sets. For more examples of Riesz sets we encourage the reader to look in \cite{MR941245}.

\section{Some Examples of Totally Predictive Sets}\label{section: some examples}
\subsection{The union of finitely many lacunary sets}
We begin by studying the union of lacunary sets. The following argument is based on one by Katznelson \cite[Section 1.2]{MR1832446}. 

\begin{proposition}\label{prop:lacunary_predictive}
	Let $\Lambda$ be a finite union of lacunary sets. Then $\N\setminus \Lambda$ is totally predictive. 
\end{proposition}
In the following we will denote the identity of the group $\mathbb T^d$ by $0$.
\begin{proof} Observe that if $\Lambda$ is a finite union of lacunary sets then so is $k+\Lambda$ for all $k \in \N$ (possibly for different constants). Thus it is sufficient to prove that the complement of a finite union of lacunary sets is predictive.

	To see this notice that it will be sufficient to prove that there exists $d\in \N$, $\alpha\in \mathbb T^d$ and a neighbourhood $U\subset \mathbb T^d$ containing $0$ such that $N(0, U)\subset \N\setminus \Lambda$. By Corollary \ref{corollary:predictive_rotation} this would imply that $\N\setminus \Lambda$ is predictive. It was in fact proved in \cite{MR540398, MR612195} and also in \cite{MR1832446} that $d$ can be taken to be $1$ for a lacunary set and the Hausdorff dimension of such $\alpha$ is $1$ but for our purposes this weaker (and simpler) result is sufficient. 
	
	Let us assume for now that $\Lambda$ is a lacunary set with a constant $\rho>4$ and $d=1$. Let 
	$$A_i:=\{\alpha\in \mathbb{T}~:~\lambda_i \alpha \notin (-1/4, 1/4)\}.$$
	Each $A_i$ is a union of closed intervals of length $1/2\lambda_i$ centred at $(2k+1)/2\lambda_i$ for $k=0,1, \ldots, \lambda_i-1$. Since $2/\lambda_{i+1}<1/2\lambda_i$ we get that each interval in $A_i$ contains at least two $\lambda_{i+1}$-roots of unity and at least one interval from $A_{i+1}$. Thus $\cap A_i$ is non-empty and $\alpha\in \cap_i A_i$ satisfies the conclusion of the theorem.
	
	Now let $\Lambda$ be a finite union of lacunary sets with constants $\rho_1, \rho_2, \ldots, \rho_n$. Take $d$ large enough such that $\rho_i^d>4$ for all $1\leq i \leq n$. Arranging elements of each lacunary set in increasing order $\lambda_1, \lambda_2, \ldots$, the sets $\{\lambda_{i+kd}~:~ k\in \N\}$ are lacunary with constant greater than $4$. Thus we can write $\Lambda$ as a union of lacunary sets $\cup_{i=1}^{nd} \Lambda_i$ such that each $\Lambda_i$ is lacunary with constant greater than $4$. For each $1\leq i \leq nd$ there exists $\alpha_i\in \mathbb T$ such that $\lambda \alpha_i \notin (-1/4, 1/4)$ for all $\lambda\in \Lambda_i$. Therefore if
	$\mathbf\alpha= (\alpha_1, \alpha_2, \ldots, \alpha_{nd})$ then $$\lambda \mathbf \alpha\notin (-1/4, 1/4)^{nd} $$ for all $\lambda\in \Lambda$
	and
	$$N(0, (-1/4, 1/4)^{nd})\subset \N\setminus \Lambda.$$
	This completes the proof.\end{proof}
We can now extract from the proof above its main principle.
\begin{proposition}\label{prop:predict using AP}
	Let $R \subset \Z$ and $P \subset \N$ be sets satisfying the following.
	\begin{enumerate}[(i)]
		\item All zero entropy processes $X_\Z$ are measurable with respect to $X_{R}$.
		\item For all $r\in R$ there exists a predictive set $P(r)\subset \N$ such that $r+P(r)\subset P$.
	\end{enumerate}
	Then $P$ is totally predictive.
\end{proposition}

\begin{proof}
	Let $X_\Z$ be a zero entropy process. Since $P(r)$ is predictive for all $r\in R$ we have that $X_r$ is measurable with respect to $X_{r+ P(r)}$. Thus $X_R$, and hence, $X_\Z$ is measurable with respect to $X_P$. Thus $P$ is totally predictive. 
\end{proof}

Since $k \N$ is a predictive set for all $k \in \N$ we get the following corollary. We will use it repeatedly in our constructions.

\begin{corollary}\label{cor:predict using AP}
Let $R \subset \Z$ and $P \subset \N$ be sets satisfying the following.
\begin{enumerate}[(i)]
	\item All zero entropy processes $X_\Z$ are measurable with respect to $X_{R}$.
	\item For all $r\in R$ there exists $a_r\in \Z\setminus \{0\}$ such that $r+\N a_r \subset P$. 
\end{enumerate}
Then $P$ is totally predictive.
\end{corollary}

\subsection{Some totally predictive sets}

\begin{theorem}\label{thm:some totally predictive sets}
	The set $\N\setminus Q_i$ is totally predictive for the following sets $Q_i$; $i=1, 2, 3, 4$:
	\begin{enumerate}[(i)]
		\item The set of perfect squares, $Q_1:=\{n^2~:~n \in \N\}$.
		\item $Q_2$ is the set of prime numbers. 
		\item Sets $Q_3\subset \N$ for which there exist an increasing sequence $n_k$ and a sequence $l_k$ such that , $n_k-l_k\to \infty$ and $Q_3\setminus \cup_{t\in \N}[tn_k, tn_k+l_k]$ is finite for all $k \in \N$. 
		\item $Q_4:=\text{IP}(n_1, n_2, \ldots)$ where $n_i\in \N$ is a sequence such that $\frac{n_{i+1}}{n_i}\in \N$ is greater than or equal to $3$.
	\end{enumerate}
	\end{theorem}

\begin{proof} Fix $t\in\N$. Since $R=\{t, t+1, \ldots \}$ is predictive we have that $X_\Z$ is measurable with respect to $X_{-R}$ for all zero entropy processes $X_\Z$. By Corollary \ref{cor:predict using AP} it is sufficient to show that for all $r \in -R$ there exists $e_r \in \N$ such that
$-r+e_r\N\subset \N\setminus Q_i$ to prove that $\N\setminus Q_i$ is a totally predictive set.

Let $a_r:=3r^2$ for all $r \in \N$. We will prove that $-r+a_r\N\subset \N\setminus Q_1$. Suppose not. Then there is $k \in \N$ such that $-r+3r^2k=r(3rk-1)$ is a perfect square. Since $r$ and $3rk-1$ are prime to each other we have that $3rk-1$ is also a perfect square but this is impossible because it is $-1$ modulo $3$. Thus $\N\setminus Q_1$ is totally predictive.

Let $b_r:=3r$ for all $r \geq 2$. Then we have that $-r+k3r=r(3k-1)$ is never a prime number. Thus we have for $r\geq 2$, $-r+b_r\N\subset \N\setminus Q_2$ and $\N\setminus Q_2$ is totally predictive.

Fix $r\in \N$. Choose $k\in \N$ such that $n_k-l_k>r$. Now for all $t \in \N$ we have that 
$$l_k+n_kt <-r+n_k(t+1)<n_k(t+1).$$ 

Since $Q_3\setminus \cup_{t\in \N}[tn_k, tn_k+l_k]$ is finite there exists $t'\in \N$ such that for all $t\geq t'$, $-r+tn_k\notin Q_3$. Let $c_r=t'n_k$. Then we have that $-r+c_r\N\subset \N\setminus Q_3$ showing that $\N\setminus Q_3$ is totally predictive.

To prove that $\N \setminus Q_4$ is totally predictive, it is sufficient to prove that there exists a sequence of natural numbers $l_k$ such that $n_k-l_k\to \infty$ and $Q_4\setminus \cup_{t\in \N}[tn_k, tn_k+l_k]$ is finite for all $k \in \N$. Let $l_k=n_1+n_2+\ldots n_{k-1}$. 

\begin{enumerate}[(i)]
	\item Clearly 
	$$n_k-l_k > n_k-\sum_{i=1}^{k-1}\frac{1}{3^i}n_k>\frac{1}{2}n_k\to \infty$$ as $k \to \infty$.
	\item Since $n_i|n_{i+1}$ for all $i \in \N$, if $t\in Q_4$ and $t> n_1+ n_2+ \ldots +n_{k-1}$ then $t=rn_k + s$ where $r\in \N$ and $0\leq s\leq n_1+ n_2+ \ldots +n_{k-1}$. Thus we have that 
	$$Q_4\setminus \cup_{t\in \N}[tn_k, tn_k+l_k]\subset \{1, 2 , \ldots ,n_1+ n_2+ \ldots +n_{k-1}\}$$ is a finite set. 
\end{enumerate}
This completes the proof.
\end{proof}
\section{The Squares and Other Sparse Sequences} \label{section:squares are sparse}
Corollary \ref{cor:return time predictive} states that return-time sets are predictive. The converse fails dramatically.

\begin{proposition}\label{prop:return not predictive}
There are totally predictive sets which do not contain a return-time set. 
\end{proposition}

To prove this we need to consider the dual of return-times sets and recall related results: A set $N\subset \N$ is called a \emph{Poincar\'{e} set} (also called first return sets) if it intersects every return-time set. It is well-known that the following sets are Poincar\'{e} sets \cite[Page 72]{MR603625}:

\begin{enumerate}[(i)]
	\item $p(\N)$ where $p$ is an integer polynomial which has a root modulo $m$ for every $m$.
	\item $(M-M)\cap \N$ where $M$ is an infinite set.
\end{enumerate}

Following the method outlined in \cite[Pages 70-72]{MR603625} we arrive at a well-known converse to the fact that the first example is a Poincar\'{e} set. 
\footnote{There are some minor typos in \cite[Pages 70-72]{MR603625}; the condition that the polynomial $p$ has a root in $\N$ should be replaced by the condition that $p$ has a root modulo $m$ for all $m \in \N$.}

\begin{theorem}\label{theorem:n_square_predictive}
	Let $p$ be a polynomial such that $p(\N)\subset \N$. Then $p(\N)$ is a Poincar\'{e} set if and only if $p$ has a root modulo $m$ for all $m\in \N$.
\end{theorem}
Thus $\{n^k; n\in \N\}$ is a Poincar\'{e} set for all $k\in \N$. This condition on integer polynomials goes back at least to \cite{MR516154}; polynomials satisfying this condition are called \emph{intersective}. It is obvious that polynomials with an integer root are intersective but there are polynomials which do not have integer roots and are still intersective: For instance, consider the polynomial $(x^2-13)(x^2-17)(x^2-221)$ \cite[Page 3]{MR0195803}. However it is well-known that irreducible polynomials (of degree greater than 1) are not intersective \cite{MR3556829,gegner2018notes}.

\begin{proof}[Proof of Proposition \ref{prop:return not predictive}]
	By Theorem \ref{thm:some totally predictive sets}, we know that $\N\setminus \{n^2~:~n \in \N\}$ is totally predictive. By Theorem \ref{theorem:n_square_predictive} the set $\{n^2~:~n \in \N\}$ is a Poincaré set and intersects every return-time set. Thus $\N\setminus \{n^2~:~n \in \N\}$ is a totally predictive set which does not contain a return-time set.
\end{proof}

\begin{remark}
Infinite difference sets are also Poincaré sets. We could have proved Proposition \ref{prop:return not predictive} by finding an infinite difference set whose complement is totally predictive. We have in fact already demonstrated such a set. For instance the complement of $\Delta(1, 1+3, 1+3+3^2, \ldots )$ is totally predictive. This follows from Theorem \ref{thm:some totally predictive sets} since
	$$\N\setminus IP(1, 3, 3^2, \ldots )\subset \N\setminus \Delta(1, 1+3, 1+3+3^2, \ldots ).$$
\end{remark}

\begin{remark}
To prove that $\N\setminus \{n^2~:~n \in \N\}$ is totally predictive, we constructed for every $n \in \N$ a return-time set in $P_n$ such that $-n+P_n\subset \N\setminus \{n^2~:~n \in \N\}$. The following question arises naturally.
	\begin{question}
	Let $N_i; i\in \N\cup\{0\}$ be a sequence of return-time sets. Enumerate the elements of $N_0$ in the increasing order: 
	$$N_0:=\{n_1, n_2, \ldots, \}.$$
	Does the set $\cup_{i \in \N}(n_i+N_i)$ necessarily contain a return-time set?
\end{question}
\end{remark}

In general one may pursue the method of obtaining new predictive sets from some given predictive sets as mentioned in Proposition \ref{prop:predict using AP}. However this method has some limitations: Repeated applications of Proposition \ref{prop:predict using AP} can't be used to prove that the set $\N\setminus \{n^3~:~n \in \N\}$ is predictive. 

\begin{question}\label{question:distance_increase}
Are $\Delta_r^\star$-sets always totally predictive? In particular (see Remark \ref{remark:about deltastar}) suppose $n_i$ is a strictly increasing sequence such that $n_{i+1}-n_i$ is also strictly increasing. Is the set $\N\setminus\{n_i~:~i \in \N\}$ totally predictive? Is the set $\{n_i~:~i \in \N\}\cup (\Z\setminus \N)$ a Riesz set? We emphasise that we do not know this even for $n_i=i^k$ where $k\geq 3$ is odd (for even integers it follows from Theorems \ref{thm:some totally predictive sets} and \ref{thm:totally predictive open is Riesz}).
\end{question}

There are a few partial results for this question. In \cite{wallen1970fourier}, Wallen gave a short and ingenious proof of the fact that if $\mu$ is a measure such that $\supp(\hat\mu)\cap\N\subset \{n_i~:~i \in \N\}$ where $n_i$ is a sequence as above then $\mu \star \mu$ is absolutely continuous.

For the case when $n_i=i^k$ for odd $k\geq 3$ we prove a partial result in the following proposition and its corollary using a suggestion by Elon Lindenstrauss.

\begin{proposition}
	\label{prop:elon} Let $Q\subset \N$ be a non-empty set such $\N\setminus Q$ is $\Delta_3^\star$. Then there does not exists a singular probability measure on $\mathbb{T}$ such that $\supp(\hat\mu)\subset Q\cup(-Q)\cup\{0\}$.
\end{proposition}
For a subset $Q\subset \N$ recall the notation 
$$Q_n:=Q\cap\{1,2,\ldots, n\}.$$ 
\begin{proof}
	Let $\mu$ be a probability measure on $\mathbb{T}$ such that $\supp(\hat \mu)\subset Q\cup(-Q)\cup\{0\}$. Let $c$ be a trigonometric polynomial $c(t):=\sum_{j\in Q} a_j e^{2\pi i jt}$ (where the sequence $a_j$ is eventually zero). It follows from the Cauchy-Schwarz inequality that
	$$|\sum_{j\in Q}\overline{ a_j }\hat{\mu}(j)|^2= |<1, c>|^2\leq \|1\|_{L^2(\mu)}^2\|c\|^2_{L^2(\mu)}\leq \sum_{j, j'\in Q} a_j \overline{a_{j'}} \hat{\mu}(j'-j).$$
	By Remark \ref{remark:about deltastar} we know that $Q+Q$ is disjoint from $Q$. It follows that the right hand side is equal to $\sum_{j\in Q}|a_j|^2$. Applying the inequality thus obtained to the polynomial $c(t):=\sum_{j\in Q_n} \hat{\mu}(j) e^{2\pi i jt}$ we get that
	$$\sum_{s\in Q_n} |\hat{\mu}(s)|^2\leq 1$$
	for all $n\in \N$. This proves that $\mu$ is not singular with respect to Lebesgue. 
\end{proof}

\begin{corollary}
	Let $k \geq 2$ and $Q:=\{n^k~:~n \in \N\}$. Then there does not exists a singular probability measure on $\mathbb{T}$ such that $\supp(\mu)\subset Q\cup(-Q)\cup\{0\}$.
\end{corollary}

\begin{proof}
	For even integers $k$ this result follows from Theorem \ref{thm:some totally predictive sets} and Proposition \ref{prop: singular measure prob riesz}. For odd $k$, it follows from Fermat's last theorem \cite{MR1333035,MR1333036} that $Q-Q$ is disjoint from $Q$. Thus by Proposition \ref{prop:elon} there does not exists a singular probability measure on $\mathbb{T}$ such that the support of its Fourier coefficients are contained in $Q\cup(-Q)\cup\{0\}$.
\end{proof}

\section{Predictive sets for $\Z^d$-actions}\label{section:zd}
In this section we will see how our principal results extend to $\Z^d$ actions. Since the proofs are essentially the same we will not repeat the details but just indicate where the key differences lie. First let us establish some notation.

For us \emph{random fields} will mean translation-invariant random fields with finite state space (unless otherwise mentioned). To begin we will first recall a formula for entropy of $\Z^d$ actions. Fix $d\geq 1$. We will use $\mathbf 0$ to denote the origin in $\Z^d$. A \emph{total invariant order} $<$ on $\Z^d$ is a total order such that if $i<j$ then $i+ k< j+ k$ for all $ k\in \Z^d$. Fix a total invariant order on $\Z^d$. For a given vector $i$ we write 
$$i^{(-)}:=\{j\in \Z^d~:~j<i\}$$
to denote the set of elements less that $i$. 

For a random field $(X_i)_{i\in \Z^d}$ and $P\subset \Z^d$, let $X_P$ denote the collection of random variables $(X_i)_{i\in P}$. Abusing notation we will also use $X_{\Z^d}$ to denote the random field itself. The entropy of the random field is given by the formula
$$h(X_\zd):=H(X_\0~|~X_{\0^{(-)}}).$$
This was established for the lexicographic ordering in \cite{MR0316680} but goes through without changes for any total invariant order. It is not difficult to see that total invariant orders on $\Z^d$ are determined by cones $C\subset \Z^d\setminus \{\0\}$ for which $C$, $-C$ and $\{\0\}$ form a partition of $\zd$. Here $C$ is the set of elements less that $\0$, $-C$ is the set of elements which are greater than $\0$ and invariance determines all other relations. 

Thus if the random field $X_{\zd}$ has zero entropy then $H(X_\0~|~X_{\0^{(-)}})=0$. A set $P\subset \0^{(-)}$ is called \emph{$\zd$-predictive} if for all zero entropy random fields $X_\zd$, $H(X_\0~|~X_{P})=0$. We now prove that the restriction of a $\zd$-predictive set to a subgroup is a predictive set for that subgroup.

\begin{proposition}
	Let $P$ be $\Z^d$-predictive and $\phi: \Z^{d'}\to \Z^d$ be an injective group homomorphism. Then $\phi^{-1}(P)$ is $\Z^{d'}$-predictive (with an appropriate invariant total order).
\end{proposition}

\begin{proof}
	Let $X'_{\Z^{d'}}$ be a zero entropy random field. Let $X^{\overline i}_{\Z^{d'}}; \overline i \in \Z^{d}/ \phi(\Z^{d'})$ be independent copies of the process $X'_{\Z^{d'}}$ and $\psi: \Z^d\to \Z^{d'}$ be a group homomorphism such that $\psi\circ \phi$ is the identity map. Given $i\in \zd$, denote by $\overline{i}\in \Z^{d}/ \phi(\Z^{d'})$ the image under the quotient map. Now consider the process $X_{\zd}$ given by 
	$$X_i:=X^{\overline i}_{\psi(i)}.$$

	Take the total invariant order on $\Z^{d'}$ where $i<_{\Z^{d'}}j$ if and only if $\phi(i)<_{\zd}\phi(j)$. Let us denote by $\0_d$ and $\0_{d'}$ the origin in $\zd$ and $\Z^{d'}$ respectively. It follows that $\phi^{-1}({\0_d}^{(-)})=\0_{d'}^{(-)}$.
	
	Since $X_{\phi(\Z^{d'})}$ is independent of $X_{\Z^d\setminus \phi(\Z^{d'})}$ we have
	$$H(X_{\0_d}~|~X_{{\0_d}^{(-)}})=H(X_{\0_d}~|~X_{{\0_d}^{(-)}\cap \phi(\Z^{d'})} ) = H(X'_{\0_{d'}}~|~X'_{\phi^{-1}({\0_d}^{(-)})})=H(X'_{\0_{d'}}~|~X'_{{\0_{d'}}^{(-)}})=0.$$
	
	This proves that $X_{\zd}$ is a zero entropy process. In a similar vein we have that
	$$0=H(X_{\0_d}~|~X_P)=H(X_{\0_d}~|~X_{P\cap \phi(\Z^{d'})})=H(X'_{\0_{d'}}~|~X'_{\phi^{-1}(P)}).$$
	Thus $\phi^{-1}(P)$ is $\Z^{d'}$-predictive.
\end{proof}
Given a $\zd$ probability preserving action $(Y, \mu, T)$ and $U\subset Y$ such that $\mu(U)>0$, we write 
$$N_\zd(U,U):=\{i\in \0^{(-)}~:~\mu(T^i(U)\cap U)>0\}.$$
Such a set is called a \emph{$\zd$-return-time set}. Similarly given $i_1, i_2, \ldots\in \zd\setminus \{\0\}$, we write 
$$\text{SIP}^-(i_1, i_2,\ldots):=\{\sum_{k\in \N}\epsilon_k i_k~:~ \epsilon_i=0,1,-1\}\cap \0^{(-)}.$$
Such sets are called SIP sets. A set $P\subset \0^{(-)}$ is called $\SIP$ if it intersects every SIP set.

The facts stated below follow as in the case for $\Z$-actions. 
\begin{enumerate}[(i)]
	\item 	$\zd$-return-time sets
	\begin{enumerate}[(a)]
		\item $\zd$-return-time sets are predictive.
		\item If $P$ is a $\zd$-predictive set and $Q$ is a $\zd$-return-time set of a zero entropy system then $P\cap Q$ is also $\zd$-predictive.
		\item Let $P$ be a $\zd$-predictive set, $Y$ a compact group and $T: \Z^d\times Y\to Y$ act on $Y$ by rotations. If $ x \in U\subset Y$ and $U$ is open then
		$$P\cap\{i\in \zd~:~T^i(x)\in U\}$$
		is $\zd$-predictive.
		\item  If $P$ is a $\Z^d$-predictive set then $\zd\setminus(P\cup (-P))$ cannot contain the support of the Fourier coefficients of a singular probability measure on $\mathbb T^d$.
	\end{enumerate}
	\item SIPs
	\begin{enumerate}[(a)]
		\item $\zd$-predictive sets are $\SIP$.
		\item $\zd$-predictive sets have bounded gaps.
		\item Given any SIP set $P\subset \Z^d$, there exists a weak mixing zero entropy $\zd$-Gaussian random field for which $X_{\0}$ is independent of $X_{i}$ for all $i\notin P$.
	\end{enumerate}
\item  Linear predictivity
\begin{enumerate}[(a)]
	\item If $X_{\Z^d}$ is a $\zd$-Gaussian random field and $P\subset \Z^d$. Then $\mathbb E(X_0~|~X_P)=X_0$ if and only if $X_0$ belongs to the linear span of $X_P$. In other words, $X_P$ can predict $X_0$ if and only if $X_P$ can linearly predict $X_0$. 
	\item Let $X_{\Z^d}$ be a real-valued (possibly infinite-valued) process on a probability space $(\Omega, \mathcal B, \mu)$ such that $X_{\0}\in L^2(\mu)$. If the spectral measure of $X_{\Z^d}$ is singular and $P$ is a predictive set then $X_P$ can linearly predict $X_{\0}$. 
\end{enumerate}
\end{enumerate}

To prove that $\zd$-return-time sets are $\zd$-predictive exactly the same ideas apply. To prove that $\zd$-predictive sets are $\SIP$ and linear predictivity we will need to use the theory of $\zd$-Gaussian random fields.

A \emph{$\zd$-Gaussian random field} is a complex valued stationary process $X_\zd$ such that $X_{i_1}, X_{i_2}, \ldots, X_{i_m}$ is jointly Gaussian for all distinct $i_1, i_2, \ldots, i_m\in \zd$. We shall assume throughout that $X_i$ have mean zero for all $i \in \zd$. Since jointly distributed Gaussian random variables are determined by their covariance, $\zd$-Gaussian random fields are determined by their autocorrelation sequence: $\mathbb E(X_0\overline{X_j}); j\in \zd$ and thereby the corresponding spectral measure $\mu$ on $\mathbb T^d$ for which 
$$\int_{\mathbb T^d}e^{2\pi i <j,x>}d \mu(x)= \mathbb{E}(X_0\overline{X_j})$$
where $<\cdot, \cdot>$ is the usual inner product. Further every finite measure $\mu$ on $\mathbb{T}^d$ uniquely determines a $\zd$-Gaussian random field with spectral measure $\mu$.
We now mention some facts about $\zd$-Gaussian random fields $X_{\zd}$ which we need to prove the facts stated above.
\begin{enumerate}[(i)]
	\item A $\zd$-Gaussian random field is ergodic if and only if the spectral measure is continuous if and only if it is weak mixing (the proofs follow the same strategy as in \cite[Pages 191 and 368]{MR832433}).
	\item A $\zd$-Gaussian random field has zero entropy if and only if the spectral measure is singular \cite[Theorem 2.1]{MR1231420}. The proof for the direction that we need (if the spectral measure is singular then the $\zd$-Gaussian random field has zero entropy) can also be found in \cite[page 3]{MR1209421}.
\end{enumerate}

\bibliographystyle{alpha}
\bibliographystyle{abbrv}
\bibliography{predictive}

\end{document}